\documentclass[reqno,11pt]{amsart}
\pdfoutput=1

\usepackage[utf8]{inputenc}
\usepackage[italian,english]{babel}
\usepackage{amsfonts}
\usepackage{amssymb}
\usepackage{tikz}
\usepackage{svg}
\usepackage[right=3cm,bottom=2.5cm,footskip=30pt]{geometry}
\usepackage[autostyle]{csquotes}
\usepackage{color}
\usepackage{mathrsfs}
\usepackage{mathtools}
\usepackage{enumitem}
\usepackage{mathrsfs}
\usepackage{nicefrac}
\usepackage{multirow}
\usepackage{yfonts}

\usepackage{resmes} % simbolo per misure ristrette 

\usetikzlibrary{svg.path,arrows,calc,decorations.pathreplacing}
\tikzset{>=latex'}

\usepackage{todo}

\usepackage[toc]{appendix}
\usepackage{etoolbox}
\cslet{blx@noerroretextools}\empty
\usepackage[maxbibnames=9,giveninits=true,style=ext-alphabetic,backend=biber]{biblatex}
\usepackage{bibsettings}
\usepackage[colorlinks, citecolor=citegreen, linkcolor=refred,urlcolor=blue, unicode,psdextra,hypertexnames=false]{hyperref}
\usepackage{cleveref}
\crefname{enumi}{}{}
\crefname{enumii}{}{}
\expandafter\def\csname ver@etex.sty\endcsname{3000/12/31}

\usepackage{headings}
\usepackage{accents}
\usepackage{autonum}
\usepackage{gasymbols}

\usepackage[dvipsnames]{xcolor} % colori extra

\definecolor{citegreen}{rgb}{0,0.3,0}
\definecolor{refred}{rgb}{0.5,0,0}
\usepackage[equation=section]{theorems}

\author[M.~Fogagnolo]{Mattia Fogagnolo\orcid{0000-0002-5933-1344}}
\address{M.~Fogagnolo, Dipartimento di Matematica, Universit\`a di Padova, via Trieste 63, Padova, Italy}
\email{mattia.fogagnolo@unipd.it }

\author[G.~Gatti]{Giorgio Gatti\orcid{0009-0006-9006-3575}}
\address{G.~Gatti, Dipartimento di Matematica, Universit\`a di Padova, via Trieste 63, Padova, Italy}
\email{giorgio.gatti@math.unipd.it}

\author[A.~Pluda]{Alessandra Pluda\orcid{0000-0003-4714-4119}}
\address{A.~Pluda, Dipartimento di Matematica, Universit\`a di Pisa,
Largo Pontecorvo 5, Pisa, Italy}
\email{alessandra.pluda@unipi.it}

\title[Scalar curvature bounds for 3D continuous metrics through the IMCF]{Scalar curvature bounds for 3D continuous metrics through the Inverse Mean Curvature Flow}
\date{\today}

\begin{document}

\begin{abstract}
We propose a notion of scalar curvature lower bounds in a three-dimensional Riemannian manifold endowed with a $\CS^0$ metric based on the monotonicity of the Hawking mass along the inverse mean curvature flow. We present a stability theorem for continuous Riemannian metrics with nonnegative scalar curvature in such IMCF sense.
\end{abstract}  
\maketitle


\smallskip

\noindent \underline{\smash{Keywords}}: scalar curvature, inverse mean curvature flow, Hawking mass, monotonicity formula.

\section{Introduction and main result}
Under the inverse mean curvature flow, a family of hypersurfaces evolves with a speed inversely proportional to their mean curvature. Its weak formulation involves the level sets of a locally Lipschitz solution $u$ to a boundary value problem. For a bounded domain $E$ within a complete noncompact Riemannian manifold $(M,g)$, $u$ must satisfy:
\begin{equation}\tag{IMCF}\label{eq:IMCF}
\begin{cases}
\div\left(\frac{\nabla u}{\abs{\nabla u}}\right)& =& \abs{\nabla u} &\text{on $ M\smallsetminus E $}\\
u&=&0&\text{on $\partial E$}\\
u&\to& +\infty &\text{as $ d(x,E)\to+\infty$}.
\end{cases}
\end{equation}
Solutions to \cref{eq:IMCF} are understood in the nonstandard variational sense introduced and comprehensively developed by Huisken and Ilmanen \cite{huisken_inversemeancurvatureflow_2001}. With the aim of proving the Riemannian Penrose inequality, they show that, in topologically simple $3$-manifolds with nonnegative scalar curvature, the Hawking mass 
\begin{equation}\label{eq:Hawking-mass}
    \ma_H(\Sigma) =\sqrt{\frac{\abs{\Sigma}}{16\pi}}\left(1-\frac{1}{16\pi}\int_{\Sigma}\H^2\dif\sigma\right)
\end{equation} 
is monotone nondecreasing. This note stems from observing that the converse holds:  
\begin{center}
\emph{in dimension $3$, the monotonicity of the Hawking mass implies nonnegative scalar curvature.} 
\end{center}
Consequently, nonnegative scalar curvature can be characterized purely via the IMCF - a formulation that makes sense even for merely continuous Riemannian metrics. 
In fact, one can consistently say (see \Cref{prop:characterization} and \Cref{def:imcf}) that a noncompact  manifold endowed with a continuous Riemannian metric admitting a weak solution to \eqref{eq:IMCF} has nonnegative scalar curvature (in the sense of IMCF) if, for all $t\geq 0$, the following holds
\begin{equation}
\label{eq:scalarimcfintro}
   32\pi\left(e^{t/2} -1\right) - \int_{\{u \leq t\}} e^{\frac{u}2} \abs{\nabla u}^3 \dif\mu \geq t\left(16\pi -\int_{\partial E} \H^2 \dif\sigma\right).
    \end{equation}
Since a $\CS^0$ metric is not enough to define the mean curvature classically, the mean curvature  $\H$ of the initial hypersurface $\partial E$ is understood in the variational sense of \Cref{def:continuous-phi-bubble}.

In this note, we present a definition of nonnegative scalar curvature in the IMCF sense and give the main arguments for its $\mathscr{C}^0$ stability.
\begin{theorem}
\label{thm:c0-stability}
   Let $M$ be a smooth, oriented,  noncompact $3$-manifold without boundary and with $H_2(M) = \{0\}$. Let $\{g_j\}_{j \in \mathbb{N}}$ be a sequence of complete $\CS^0$-Riemannian metrics on $M$ and assume that each  $(M, g_j)$ satisfies a  global Euclidean  isoperimetric inequality \cref{eq:isop} with a constant $C_{\mathrm{iso}}$ independent of $j$ and has nonnegative scalar curvature in the sense of IMCF. Let $g_j \to g$ as $j \to \infty$ in $\CS^0_{\mathrm{loc}}(M)$. Then, $(M, g)$ has nonnegative scalar curvature in the sense of IMCF as well.
\end{theorem}
 For sequences of smooth metrics, a $\mathscr{C}^0$-stability result was first established using entirely different techniques by Gromov \cite{gromov-billiards} (see also \cite{chaoli-polyhedron}), and refined through a completely different Ricci flow argument by Bamler \cite{bamler}. Burkhardt-Guim \cite{Paula1, Paula2} devised  a Ricci-flow related notion of scalar curvature bounds for $\mathscr{C}^0$ metrics that she showed to be stable.    
 
Defining and understanding scalar curvature bounds in a non-smooth setting may also be of interest outside of Riemannian metric geometry. In fact, since nonnegative scalar curvature arises when constraining the Dominant Energy Condition of a Lorentzian spacetime to an initial Riemannian dataset, the ideas put forth here are naturally tied to the very active field of non-smooth spacetime geometry, where $\mathscr{C}^0$ regularity plays a crucial role (see the very recent works \cite{mondino-ryborz-samann, beran-harvey-samann} and references therein).

We point out that this note has a partly expository nature. We focus here on the core geometric ideas and on the main steps of the proofs, deferring the many technical details to \cite{fgp1} which will contain a comprehensive study of the IMCF for $\mathscr{C}^0$ metrics. We emphasize that \Cref{thm:c0-stability} can be improved with a localized statement, thereby completely removing the global topological and the isoperimetric assumptions on $g_j$. The latter are, in fact, added only to deal with the more familiar global IMCFs \cite{properness_IMCF_kaixu}. The localization can be obtained by replacing the global IMCF with a suitable local version, inspired by \cite[Theorem 3.1]{antonelli_fogagnolo_nardulli_pozzetta_positivemass} (see also \cite{kai-outer} for a different approach in the smooth setting). Moreover, we decided to focus only on nonnegative scalar curvature to have simpler statements and computations, but the extension to more general lower bounds is easily achieved. All of these improvements will be part of a future installment of this work.

Finally, we must mention the very recent developments by Mazurowski-Yao \cite{strs01,strs02}, which recover and improve upon Gromov's theorem \cite{gromov-billiards} in dimension $3$ through harmonic functions (see in this direction the very inspirational \cite{amo}) and $\mu$-bubbles respectively. We also highlight the work of Lee \cite{lee-ricciflownuovo}, which strengthens the approach of Bamler \cite{bamler} to recover the results of \cite{strs01} in all dimensions. 
We hope that working on and possibly unifying  the various perspectives currently arising in this context will lead to a deeper understanding of scalar curvature bounds in low regularity.

The paper is organized as follows. \Cref{sec:setting} introduces the weak IMCF and the variational mean curvature. Next, \Cref{thm:imcfbv} establishes the existence of a weak IMCF from a finite perimeter set with $2$-dimensional boundary (see \Cref{def:2dboundaries}) in a noncompact manifold $M$  endowed with a $\CS^0$-Riemannian metric and satisfying a global Euclidean isoperimetric inequality. In \Cref{sec:hawkingscalar} we first show the equivalence between nonnegative scalar curvature and \cref{eq:scalarimcfintro} in \Cref{prop:characterization}, in order to introduce the definition of nonnegative scalar curvature in the sense of IMCF. Finally, we sketch the proof of \Cref{thm:c0-stability}. 

\textbf{Acknowledgments.} The authors are grateful to Virginia Agostiniani, Luca Benatti, Camillo Brena, Paula Burkhardt-Guim, Esther Cabezas-Rivas, Emanuele Caputo, Luca Gennaioli, Robert Haslhofer, Gerhard Huisken, Lorenzo Mazzieri, Marco Pozzetta, Fabian Rupp, Felix Schulze, Peter Topping and Ivan Y. Violo  for their interest and very fruitful interactions. M. F. is supported by the STARS project “DEFORM” of the University of Padova.  M.F. and G.G. are partially supported by the Project “GIANTS” funded by INdAM. All authors are members of GNAMPA - INdAM and partially supported by the project “ANGELS”  funded by GNAMPA - INdAM.

\bigskip

\section{Setting and tools}\label{sec:setting}

We consider $M$ to be a smooth, oriented and noncompact $3$-manifold without boundary.  When we say that a global Euclidean isoperimetric inequality holds in $M$, for some continuous Riemannian metric $g$, we mean that there exists $\kst_{\mathrm{iso}} > 0$ such that
   \begin{equation}\label{eq:isop}
       \abs{\partial F}^{3/2} \geq \kst_{\mathrm{iso}}\abs{F}
   \end{equation}
for any bounded $F \subset M$.

When we write $\dif \sigma$ and $\dif \mu$ in the integrals, it is shorthand for $\dif \Hff^2$ and $\dif \Hff^3$, respectively.

Given a finite perimeter set $E$, we denote by $\partial^{\ast}E$ its reduced boundary and by $E^1$ and $E^0$ the measure theoretic interior and exterior respectively, i.e. the sets of points of density one and zero, respectively.

\subsection{Weak variational IMCF}

We will work in $M$ with a $\CS^0$-Riemannian metric $g$; in this section we isolate a weak notion of IMCF suitable for such a setting, which generalizes the Lipschitz formulation given by \cite{huisken_inversemeancurvatureflow_2001}, and is strongly inspired by \cite{Mazon_Segura_Dirichlet_problem_related_to_lvlset-IMCF}.
The following is the class of initial data we are working with.

\begin{definition}[$2$-dimensional boundaries]
\label{def:2dboundaries}
Let $E$ be a finite perimeter set, we say that $E$ has $2$-dimensional boundary if
    \begin{align}\label{eq:2-dim-boundary}
        \Hff^{2}(\partial E \setminus \partial^{\ast} E) =0.
    \end{align}
   
\end{definition}

\begin{definition}
\label{def:variational-IMCF}
Given a bounded open set \(E \subset M\) with finite perimeter and \(2\)-dimensional boundary, we say that a nonnegative function \(u \in \BV_{\loc}(M) \cap L^{\infty}_{\loc}(M)\) is a proper  weak variational solution to the IMCF of \(E\) if, for every bounded set of finite perimeter $K$ with $\Hff^{2}\left((K^1\cup \partial^*K)\setminus E^0\right) =0 $, $u$ minimizes the functional 
    \begin{align}\label{eq:variational-IMCF}
         J_u^K(v)  = \abs{Dv} (K^1\cup \partial^*K) + \int_{K^1 \cup \partial^*K} v \dif \abs{ Du}
    \end{align}
among all competitors $v\in \BV_{\loc}(M) \cap L^{\infty}_{\loc}(M)$ such that $v = u$ a.e. outside $K$, subject to:
 \begin{equation}
 \label{eq:weakdrichlet}
    \begin{cases}
u&=&0&\text{a.e. in  $E$},\\
\tr_{\partial E} u&=&0,&\\
\{u < t\}&\text{has}& \text{a bounded representative}&\text{for all } t\geq 0.
\end{cases}
 \end{equation}
  By $\tr_{\partial E} u =0$ we mean
\begin{equation}\label{eq:trace}
  \lim_{\rho \rightarrow 0} \,\dashint_{B_{\rho}(p) \cap M \setminus \overline{E}} \abs{u} \dif \mu = 0 \quad \text{for \(\Hff^{2}\)-a.e. }p \in \partial E.  
\end{equation}
We set 
\begin{equation}
    E_t = \mathrm{int}\{u \leq t\} \quad \mathrm{for} \, t\geq 0.
\end{equation}
\end{definition}
Under this definition, one can show, as in \cite{huisken_inversemeancurvatureflow_2001}, that for every $t\geq 0$, the set $E_t$ is strictly outward minimizing. In particular, $E_0 = E^*$, where $E^*$ is the strictly outward minimizing hull of the initial set $E$ \cite{Mazzieri_Fogagnolo_Minimising_hulls}.

\begin{remark}
When $u$ happens to be Lipschitz, it can be checked that the above notion coincides with  that of Huisken-Ilmanen \cite{huisken_inversemeancurvatureflow_2001}. In particular, when the ambient metric is smooth and the initial set $E$ is $C^2$, a proper weak variational solution to IMCF in the sense of \Cref{def:variational-IMCF} is a Huisken-Ilmanen's weak Lipschitz solution. Details on this will be included in \cite{fgp1}.
The first two conditions in \eqref{eq:weakdrichlet} naturally replace the requirement that $u$ vanishes on the boundary of $E$ in \cref{eq:IMCF}, in such a way that $E$ is prescribed to be the initial datum. Similarly, the last condition ensures that the solution is proper, analogous to requiring that $u \to +\infty$ at infinity. 
\end{remark}

\emph{For the whole duration of the note, we will refer to a proper weak variational solution to IMCF simply as weak IMCF, or weak solution to \eqref{eq:IMCF}}.

\medskip

The existence of a proper weak variational solution to the IMCF is ensured whenever the ambient metric $g$ satisfies a global isoperimetric inequality. The result holds regardless of the dimension, but we state here in dimension $3$ only.  Although conceptually simple, the proof rests on a number of rather technical steps. The full proof  will be contained in \cite{fgp1}.

\begin{theorem}
\label{thm:imcfbv}
Let $M$ be a smooth noncompact Riemannian $3$-manifold with a global Euclidean isoperimetric inequality \cref{eq:isop} and let $g$ be a complete $\CS^0$-Riemannian metric on $M$. Then, for any given bounded open set $E\subset M$ with finite perimeter and $2$-dimensional boundary, there exists a unique proper weak variational solution to the IMCF of $E$.
\end{theorem}
\begin{proof}[Sketch of the proof]
To guide the reader, we divide the proof in steps.
%that will be reflected in the presentation in \cite{fgp1}.\\

\textit{Step 1: approximation.}
    We first consider a sequence of smooth regularizing metrics $g_j$ such that $g_j\to g$ \emph{globally} uniformly in $\CS^0$ \cite[Proposition 4.9]{Lenght_structures_continuous_manifold_Burtscher}. The global uniformity in particular implies that there exists an isoperimetric constant $\kst_{\mathrm{iso}}$ common to all $g_j$'s and $g$. 

    Then, \(E\) having \(2\)-dimensional boundary allows us to apply \cite[Theorem 1.1]{strict_approximation_schmidt} to find a sequence \(E_j\) of equibounded open sets with smooth boundary which converge to \(E\) as sets of finite perimeter and satisfy \(E \subset E_j\) for all \(j \in \N\).
    .
    For each $g_j$, we can apply \cite[Theorem 1.2]{properness_IMCF_kaixu} or \cite{benatti-mari-rigoli-setti-xu} and obtain a proper, Lipschitz IMCF  $u_j$ starting from \(E_j\). \\

\textit{Step 2: $L^\infty_{\mathrm{loc}}$-estimates.}
We first crucially observe that the $u_j$'s naturally come with a uniform bounded oscillation estimate. Indeed, by \cite[Theorem 2.8]{BPP_monotonicity_formulas}, we can realize $u_j$ as the local uniform limit as $p\to 1^+$ of $u_j^p= -(p-1)\log(w_j^p)$, where $w_j^p$ is $p$-harmonic with respect to the metric $g_j$.  Then, the Harnack inequality of \cite{rigoli-salvaroti-vignati} applies to $w_j^p$, translating into 
\begin{equation}
\label{eq:p_harnack}
    u^p_j(x) \leq u^p_j (y) + \kst_K
\end{equation}
for any $K \Subset M\setminus \overline{E}_j$, with $\kst_K$ independent of $j$ and $p$. The constant $\kst_K$ can be checked to depend on $\CS^0$ quantities only.
We thus get 
\begin{equation}\label{eq:harnack}
    u_j(x) \leq u_j (y) + \kst_K
\end{equation}
for any $K \Subset M\setminus \overline{E}_j$, with $\kst_K$ independent of $j$. 

The validity of \eqref{eq:harnack} implies that if the sequence $u_j$ were not locally uniformly bounded, the level sets $\{u_j < t\}$, for $t > 0$ would collapse to $E^*$ as $j \to \infty$. 

We first show that this cannot happen for an auxiliary IMCF built as follows. Let $o \in E$, $R > 0$ small enough such that $\overline{B}_{2R}(o)$ is a closed disk satisfying $\overline{B}_{2R}(o) \subset E^1$. Let $\tilde{g}$ be a continuous metric that coincides with $g$ in $M \setminus B_{2R}(o)$ and with the flat metric $\delta$ in $B_R(o)$. Let $\tilde{g}_j$ be  uniform smooth approximations of $\tilde g$ that coincide with \(g_j\) in \(M \setminus B_{2R}(o)\) and with $\delta$ in $B_{R/2}(o)$. If $R$ is small enough \cite[Lemma 2.6.3]{Kai_thesis} ensures that for $\rho \leq R/4$ Euclidean balls $\mathbb{B}_\rho(o)$ are strictly outward minimizing with positive mean curvature for all $\tilde g_j$'s. In particular, the weak IMCFs $\tilde{u}_j$ issuing from  $\mathbb{B}_{R/8}(o)$ are all smooth and given by $u_j(x) = f(\abs{x})$, for $R/8 \leq \abs{x} \leq R/4$, which implies that $\{\tilde{u}_j < t\} \to \mathbb{B}_{R/8}(o)$ cannot happen. By construction of \(\tilde{g}_j\), the functions $\tilde{u}_j$ are actually $g_j$-IMCF outside \(E\) and they can be used as a barrier to obtain uniform $L^\infty_{\mathrm{loc}}$ bounds on $u_j$.\\

\textit{Step 3: $BV$-estimates and existence.}
% The approximating IMCFS's $u_j$ are Lipshitz weak solutions to \cref{eq:IMCF} in the sense of \Cref{def:variational-IMCF}.
Once uniform $L^\infty_{\mathrm{loc}}$-estimates on $u_j$ are available, the weak formulation itself, together with the (locally) uniform convergence of the metrics, implies $BV$-compactness for $u_j$. Up to subsequence, we can take the limit $u_j \to u$ and verify that $u$ is a proper weak variational solution to IMCF. To show the global properness of \(u\), we crucially use the fact that \((M,g)\) has a global isoperimetric inequality. In fact, it follows from the global uniform convergence \(g_j \to g\)  that a global Euclidean isoperimetric inequality holds on all \((M,g_j)\) with isoperimetric constant independent of \(j\), so one can use \cite[Theorem 4.1]{properness_IMCF_kaixu} to see that the \(u_j\) are uniformly proper and obtain that \(u\) must be proper as well.\\

\textit{Step 4: uniqueness.} The uniqueness follows from a nontrivial generalization of \cite[Uniqueness Theorem 2.2]{huisken_inversemeancurvatureflow_2001}. In the original proof, the continuity of two compared solutions \(u, v\) is required to show that if \(u,v\) are constant on \(\lbrace u > v \rbrace \subset \subset M \setminus \overline{E}\), then \(\lbrace u > v \rbrace\) must be empty. To make up for the lack of continuity, one can instead combine the coarea formula for BV functions together with a local isoperimetric inequality to show that if \(\abs{Du}(\lbrace u > v\rbrace) = \abs{Dv}(\lbrace u > v \rbrace) = 0\) then \(\abs{\lbrace  u > v\rbrace} = 0\). The rest of the proof can be carried out as in the original, taking into account the technicalities introduced by the use of $BV$ functions.
\end{proof}

\begin{remark}
The principle that the weak formulation of \Cref{def:variational-IMCF} together with uniform $L^\infty_{\mathrm{loc}}$-estimates implies $BV$-compactness is not new, although it has never been used in a $\mathscr{C}^0$ setting before. It was, in fact, used in  the classical Euclidean setting \cite{Mazon_Segura_Dirichlet_problem_related_to_lvlset-IMCF, Mazon_Segura_Non-homogeneus_elliptic_problem_lvlset-IMCF}
and the in the anisotropic setting \cite{Cabezas-Rivas_Moll_Solera_Anisotropic_IMCF}, which serve as important inspirations for this part of the proof. The attainment of $u = 0$ on $\partial E$ in the trace sense is also inspired by these papers; we have slightly improved in this direction by relaxing their assumption of Lipschitz initial boundary to requiring \cref{eq:2-dim-boundary} only.
% It benefits also of a further equivalent definition of proper weak IMCF that uses the notion of Anzellotti's pairing \cite{Anzellotti_pairing_between_measures_and_bounded_functions, Crasta_DeCicco_Anzellotti_pairing_theory_and_Gauss_Green}. 

On the other hand, the argument presented in
\textit{Step 2} seems to be novel, and genuinely motivated by the lack of both barriers at infinity and local curvature estimates in general $\mathscr{C}^0$-regular metrics.
In fact, in the aforementioned works, these estimates are derived using explicit barrier functions that we cannot provide in our context. Where such barriers are not available, as in nonflat Riemannian settings, local bounds on the Ricci curvature are  employed \cite{huisken_inversemeancurvatureflow_2001, properness_IMCF_kaixu, benatti-mari-rigoli-setti-xu}.  Again, this is of course harmless in any smooth ambient, but not available in general $\mathscr{C}^0$ contexts. 
\end{remark}

\subsection{Variational mean curvature}

\begin{definition}
\label{def:continuous-phi-bubble}
We say that an open bounded set $E \subset M$ has continuous variational mean curvature $\varphi$ if there exist two bounded sets $U_1 \Subset E$, $U_2 \Supset E$ and $\varphi \in \CS^0(M)$ such that $E$ is a minimizer of
\begin{equation}
\label{eq:phi-bubble-def}
    \abs{\partial F} -\int_F \varphi \dif\mu
\end{equation}
among all competitors $F$ with finite perimeter satisfying $U_1\subset F \subset U_2$.
\end{definition}

In a smooth ambient space,  the class of $\CS^2$ bounded sets coincides with that of sets with continuous variational mean curvature. 
\begin{lemma}
\label{lem:bubble} Let $(M,g)$ be a smooth Riemannian manifold.
Let $E \subset M$ be an open bounded set with $\CS^2$ boundary. Then, there exist $U_1 \Subset E$, $U_2 \Supset E$ and $\varphi \in \CS^0(M)$ such that $E$ is the unique minimizer of
\begin{equation}
\label{eq:phi-bubble-lemma}
    \abs{\partial F} -\int_F \varphi \dif\mu
\end{equation}
among all competitors $F$ with finite perimeter satisfying $U_1\subset F \subset U_2$. In particular, $E$ has continuous variational mean curvature $\varphi$.
\end{lemma}
\begin{proof}
Let $U^s_1 = \{x \in E \, | \, d(x, \partial E) > s\}$ and $U^s_2 = \{x \in M \setminus \overline{E} \, | \, d(x, \partial E) < s\}$ and $\Sigma_i^s = \partial U_i^s$ for $i = 1, 2$. Take $\overline{s}$ small enough so that for any $0 \leq s \leq \overline{s}$ the sets $\Sigma_1^s$ and $\Sigma_2^s$ are $\CS^2$ hypersurfaces. 
%Let $U_i = U_i^{\overline{s}}$ for $i= 1, 2$.

 For $x_s \in \Sigma_i^s$, let $\H(x_s)$ be the mean curvature of $\Sigma_i^s$ at $x_s$. Observe that the function $\H(x)$ is continuous. Define $\varphi \in \CS^0(M) \cap C^\infty(M \smallsetminus \partial E)$ so that, for all $0 < s \leq \overline{s}$ we have $\varphi(x_s) > \H(x_s)$ if $x_s \in \Sigma_1^s$ and $\varphi(x_s) < \H(x_s)$ if $x_s \in \Sigma_2^s$.  
     
By the direct method in the Calculus of Variations there exists     a minimizer $U_1^{\overline{s}} \subset \widetilde{E} \subset U_2^{\overline{s}}$ for \eqref{eq:phi-bubble-lemma} among competitors $F$ satisfying $U_1^{\overline{s}} \subset \widetilde{E} \subset U_2^{\overline{s}}$. We conclude by observing that the maximum principle forces $\widetilde{E} = E$. Namely, one can first invoke \cite[Theorem 5]{white-maximum} to see that $\partial\widetilde{E} \cap \Sigma_i^{\overline{s}} = \emptyset$ for both $i =1, 2$. By classical regularity theorems (see e.g. \cite[Theorem 2.2]{zhou-zhu} for an account) we then deduce that $\partial \widetilde{E} \smallsetminus \partial E$ is smooth, since it is disjoint from the constraining boundaries $\Sigma_i^{\overline{s}}$. Finally, by the classical maximum principle $\partial{\widetilde E} \cap \Sigma_i^s = \emptyset$ for $i =1, 2$ and $0 < s < \overline{s}$, implying $\partial \widetilde{E} = \partial E$, and concluding the proof. 
\end{proof}

\section{Hawking mass and scalar curvature bounds}
\label{sec:hawkingscalar}

Let $(M,g)$ be a smooth, complete, noncompact Riemannian $3$-manifold without boundary and consider a solution $u$ of \cref{eq:IMCF} for $(M, g)$  with $E$ of class  $\CS^2$. Then, thanks to the existence and regularity theory developed by Huisken and Ilmanen \cite{huisken_inversemeancurvatureflow_2001}, $u$ is in fact Lipschitz and its level sets are of class $\CS^{1,\alpha} \cap W^{2,2}$. Such regularity suffices for the definition of the quantity
\begin{equation}\label{eq:hawkingclassica}
m_H(t) = e^{\frac{t}{2}}\left(16\pi - \int_{\partial E_t}\H^2\dif \sigma\right),
\end{equation}
where $\H$ is the weak mean curvature of $\partial E_t$, that by the weak formulation of \cref{eq:IMCF} coincides with $\abs{\nabla u}$ almost everywhere on $\partial E_t$ for almost every \(t > 0\). 
  The quantity \cref{eq:hawkingclassica} coincides 
    with $\ma_H(\partial E_t)$  defined in \cref{eq:Hawking-mass}  up to a multiplicative constant.
Indeed, the evolution equation for the area yields $\abs{\partial E_t}=\ee^t\abs{\partial E^\ast}$, where $E^\ast$ is the strictly outward minimizing hull of $E$.

%, and in particular the following quantity, that substantially is the Hawking mass of $E_t$, is well defined 

\begin{theorem}
\label{prop:characterization}
   Let $(M, g)$ be a smooth, complete, oriented, noncompact Riemannian $3$-manifold  with $H_2(M) = \{0\}$. Assume that a global Euclidean  isoperimetric inequality \cref{eq:isop} holds.
   Then, the following are equivalent:
   \begin{itemize}
\item[(i)] $(M, g)$ has nonnegative scalar curvature; 
\item[(ii)] for any bounded  $E \subset M$  with connected boundary and continuous variational mean curvature, we have
\begin{equation}
\label{eq:hawkingintegrata}
\int_0^T m_H(t) \dif t \geq T  \left(16\pi - \int_{\partial E} \H^2 \dif\sigma \right),
\end{equation}
where $\H$ is the weak (distributional) mean curvature of $\partial E$.
\end{itemize}
\end{theorem}
\begin{remark}
    In the setting of condition $(ii)$, the regularity theory for prescribed mean curvature sets implies that, if $\varphi$ is a variational mean curvature for $\partial E$, then $\H = \varphi_{|\partial E}$.
\end{remark}
\begin{proof}
$(i) \Rightarrow (ii)$. Let $E \subset M$ be a bounded set with connected boundary and continuous variational mean curvature. By classical regularity theory, $\partial E$ is of class $\CS^{1,\alpha} \cap W^{2, p}$ for any $0 < \alpha < 1$ and $1<p<+\infty$, with continuous distributional mean curvature $\H$ (see e.g. \cite[Theorem 2.2]{zhou-zhu} for an account). Geroch monotonicity for the weak IMCF \cite[Geroch Monotonicity Formula 5.8] {huisken_inversemeancurvatureflow_2001} then yields
$
m_H(t) \geq m_H(0)$
for any $t \geq 0$. Notice that for this we need the connectedness of the level sets, which is granted by the topological assumption $H_2(M) = \{0\}$, see \cite[Proposition 1]{bray-miao} or \cite[Proposition 2.5]{benatti-fogagnolo-mazzieri}. Since $E_0 = E^*$ is the strictly outward minimizing hull of $E$, we also have  $\int_{\partial E_0} \H^2 \dif \sigma\leq \int_{\partial E} \H^2\dif \sigma$ (from, e.g. \cite[(1.15)]{huisken_inversemeancurvatureflow_2001}), and so 
\begin{equation}
    m_H(t) \geq m_H(0)\geq  \left(16\pi - \int_{\partial E} \H^2 \dif\sigma \right)
\end{equation}
for any $t \geq 0$. Integrating over $[0, T]$ provides \eqref{eq:hawkingintegrata}. 
\smallskip

$(ii) \Rightarrow (i)$. Let $p \in M$.  We aim to show that $\mathrm{Scal}(p) \geq 0$. For small enough radii $R$, $\partial B_R(p) \approx \mathbb{S}^2$ is smooth, strictly mean-convex and strictly outward minimizing (see \cite[Lemma 2.6.3]{Kai_thesis}). Furthermore, by \Cref{lem:bubble}, $B_R(p)$ has continuous variational mean curvature. Let then $E = E_0 =  B_R(p)$ and take $u$ weak a solution of \cref{eq:IMCF} with initial datum $E$. Observe that the right hand side of \eqref{eq:hawkingintegrata} coincides in this case with $m_H(0)$. By \cite[Smooth Start Lemma 2.4]{huisken_inversemeancurvatureflow_2001}, $u$ is smooth and nondegenerate up to $\partial B_R$, and in particular so is $m_H(t)$ for $t \in [0, T]$, with $T$ small enough. By \eqref{eq:hawkingintegrata}, we have
\begin{equation}
\label{eq:conto-gerochsmooth}
0 \leq 2\lim_{T \to 0^+}\frac{\int_0^Tm_H(t) \dif t  - Tm_H(0)}{T^2} = m_H'(0) = \int_{\partial B_R} \mathrm{Scal} + \abs{\mathring{\h}}^2 + 2\abs{\nabla_{\partial B_R} \log \H}^2 \dif\sigma
\end{equation}
where $\mathring{\h}$ denotes the traceless second fundamental form of $\partial B_R$. The last equality is readily obtained combining the evolution equation of the Willmore functional together with Gauss equation and Gauss-Bonnet theorem - again, connectedness is crucial here (see \cite{huisken_inversemeancurvatureflow_2001}, or \cite[Section 2.1]{benatti-fogagnolo-survey} for a direct derivation). 
%The smooth classical computation leading to the right hand side can be found in \cite{huisken_inversemeancurvatureflow_2001}; see \cite[Section 2.1]{benatti-fogagnolo-survey} for a direct derivation. 
%Observe now that 
Then from the classical Taylor expansion of a Riemannian metric in normal coordinates (see e.g. \cite[(3.4)]{schoen-yau-lectures}), one can deduce that
\begin{equation}
\label{eq:control-expansion}
\sup_{\partial B_R} |\mathring{\h}|^2 + |\nabla_{\partial B_R} \log \H|^2 \leq \kst R^2 .
\end{equation}
To see it, it suffices to recall that in geodesic polar coordinates around $p$ we have that the second fundamental form $\h$ satisfies $2\h_{ij}= \partial_r g_{ij}$, and exploit the aforementioned expansion read in these coordinates. Then, combining \eqref{eq:conto-gerochsmooth} and \eqref{eq:control-expansion} we conclude that
\begin{equation}
  0\leq \lim_{R\to 0^+}\dashint_{\partial B_R} \mathrm{Scal} + |\mathring{\h}|^2 + 2|\nabla \log \H|^2 \dif \sigma = \mathrm{Scal}(p),  
\end{equation}
as claimed. 
\end{proof}

Keeping in mind that $\H=\abs{\nabla u}$ by the definition of \cref{eq:IMCF}, through the coarea formula  the inequality \cref{eq:hawkingintegrata} directly translates into
\begin{equation}
    32\pi\left(e^{t/2} -1\right) - \int_{\{u \leq t\}} e^{\frac{u}2} \abs{\nabla u}^3 \dif\mu \geq t\left(16\pi -\int_{\partial E} \H^2 \dif\sigma\right).
\end{equation}
Moreover, if $\partial E$ has continuous, variational mean curvature $\varphi$, the first variation formula yields $\H = \varphi$.
Thus, with the availability of the weak variational IMCF for $\CS^0$ metrics, and with the characterization discussed in \Cref{prop:characterization}, we can propose the following notion of nonnegative scalar curvature in the sense of IMCF.

\begin{definition}
\label{def:imcf}
Let $M$ be a smooth, noncompact, oriented  $3$-manifold  with $H_2(M) = \{0\}$, and let $g$ be a complete $\CS^0$-Riemannian metric on $M$. Assume that a global Euclidean isoperimetric inequality \cref{eq:isop} holds in $M$. We say that $(M, g)$ has \emph{nonnegative scalar curvature in the sense of IMCF} if, for any bounded $E \subset M$ with connected boundary and continuous variational mean curvature $\varphi$, we have that the solution $u$ to the proper weak variational IMCF from $E$ is of class $W^{1, 3}_{\mathrm{loc}}(M)$  and 
\begin{equation}
\label{eq:nonnegscal-imcf}
    32\pi\left(e^{t/2} -1\right) - \int_{\{u \leq t\}} e^{\frac{u}2} \abs{\nabla u}^3 \dif \mu\geq t\left(16\pi -\int_{\partial E} \varphi^2\dif \sigma \right)
\end{equation}
holds for all $t \geq 0$.
\end{definition}

\subsection{Stability}

Alongside \Cref{prop:characterization}, a fundamental test for the proposed notion of nonnegative scalar curvature is its $\CS^0$ stability.

%\begin{theorem}
%\label{thm:c0-stability}
%   Let $M$ be a smooth, noncompact $3$-manifold  with $H_2(M) = \{0\}$, and let $\{g_j\}_{j \in \mathbb{N}}$ be a sequence of complete $\CS^0$-Riemannian metrics on $M$. Assume that all $(M, g_j)$ satisfy a  global
   %Euclidean   isoperimetric inequality with a constant $C_{\mathrm{iso}}$ independent of $j$ and have nonnegative scalar curvature in the sense of IMCF. Let $g_j \to g$ as $j \to \infty$ in $\CS^0_{\mathrm{loc}}(M)$. Then, $(M, g)$ has nonnegative scalar curvature in the sense of IMCF as well.
%if the $g_j$'s have nonnegative scalar curvature in the IMCF sense, then so does $g$.
%\end{theorem}

\begin{proof}[Sketch of the proof of \Cref{thm:c0-stability}]
   Consider $E\subset M$ a bounded set with connected boundary and continuous variational mean curvature $\varphi$. By slightly modifying $\varphi$ outside $\partial E$, we can assume that $E$ is the unique minimizer of \cref{eq:phi-bubble-def}. Taking inspiration from \cite{gromov-billiards}, one can construct a sequence of bounded sets $E_j\subset M$ with connected boundary and  continuous variational mean curvature $\varphi$  with respect to $g_j$  such that $\{E_j\}_{j\in\mathbb{N}}$ converges to $E$ as finite perimeter sets. 
   %such that for every $j$, there exists $E_j\subset M$ with variational mean curvature $\tilde\varphi$ with respect to $g_j$ satisfying $E_j\to E$ as sets of finite perimeter.One can thus find a modification of  $\varphi$, call it $\tilde\varphi$, such that $\tilde\varphi = \varphi$ on $\partial E$ and there exists $E_j$ with variational mean curvature $\tilde\varphi$ with respect to $g_j$ satisfying $E_j\to E$ as sets of finite perimeter. The availability of a similar approximation is clear from the ideas contained in \cite{gromov-billiards}. 
   The precise construction can be carried out through a perusal of \cite[Lemma 6.2]{huisken_inversemeancurvatureflow_2001} adapted to the varying $\CS^0$ metrics.
%claimed above can however be carried out through a perusal of \cite[Lemma 6.2]{huisken_inversemeancurvatureflow_2001} and an adapation to varying $\CS^0$-metrics. 
Each $E_j$ is obtained by minimizing \cref{eq:phi-bubble-def}  with respect to $g_j$; thanks to the convergence of $g_j$, the sets $E_j$ converge as finite perimeter sets to $E$, the unique minimizer  of \cref{eq:phi-bubble-def}  with respect to $g$. The variational nature of $E_j$ and $E$ also allows us to promote the convergence to Hausdorff convergence, so for $j\in \N$ sufficiently large, $U_1 \Subset E_j \Subset U_2$ and thus they have variational mean curvature \(\varphi\) with respect to \(g_j\). Moreover, the weak$^*$-convergence $\haus^2_j\resmes \partial E_j \to \haus^2\resmes \partial E$ \cite[Proposition 1.80]{ambrosiofuscopallara_BVfunctions_2000} implies that 
   \begin{equation}
\label{eq:convwill}
       \lim_{j \to +\infty}\int_{\partial E_j} \varphi^2 \dif\haus_j^2 = \int_{\partial E} \varphi^2 \dif\haus^2.
   \end{equation}
Observe that the continuity of the prescribed mean curvature $\varphi$ is crucially exploited here.
   Furthermore, one can invoke the regularity theory presented in \cite{regularity_quasiminimizers_ambrosio_paolini}: the Reifenberg-type $\CS^{0,\alpha}$-regularity estimates on $\partial E_j$ are stable under $\CS^0$-ambient metric limits. Therefore, $\partial E_j$ and $\partial E$ are homeomorphic for $j$ large enough, and in particular $\partial E_j$ is connected. 

   For every $j\in\mathbb{N}$, let $u_j$ be the unique proper weak variational solution to the IMCF of $E_j$, which exists thanks to \Cref{thm:imcfbv}. Observe that the proof of this result gives \textit{uniform in $j$} $BV \cap L^\infty$  estimates for $u_j$, hence $u_j \to u$ in $BV \cap L^p$, for any $1 \leq p < +\infty$, where $u$ can be verified to be the unique proper weak variational solution to the IMCF of $E$ in the sense of \Cref{def:variational-IMCF}. Since $(M,g_j)$ has nonnegative scalar curvature in the sense of IMCF, \eqref{eq:nonnegscal-imcf} holds for $u_j$. We stress the fact that \eqref{eq:nonnegscal-imcf} itself yields a uniform bound on $\int_{\{u_j \leq t\}}\abs{\nabla u_j}^3\dif \mu_j$ because of \eqref{eq:convwill}.
   %Since $(M,g_j)$ has nonnegative scalar curvature in the sense of IMCF, \eqref{eq:nonnegscal-imcf} holds for $u_j$. The analysis carried out in the proof of such result actually shows that the $BV \cap L^\infty$  estimates for $u_j$ are uniform in $j$. This implies that $u_j \to u$ in $BV \cap L^p$, for any $1 \leq p < +\infty$, and where $u$ is checked to be the (unique) IMCF of $E$ in the sense of \cref{def:imcf}. Moreover, \eqref{eq:nonnegscal-imcf} itself yields a uniform bound on $\int_{\{u_j \leq t\}}|\nabla u_j|^3\dif \mu$. 
   All these pieces of information combined result in $u \in W^{1,3}_{\loc}(M)$ and in the lower semicontinuity
   \begin{equation}
\label{eq:lschawking}
       \int_{\{u\leq t\}} e^{\frac{u}2}\abs{\nabla u}^3 \dif\mu \leq \liminf_{j \to +\infty} \int_{\{u_j \leq t\}} e^{\frac{u_j}2}\abs{\nabla u_j}^3 \dif\mu_j.
   \end{equation}
The proof is then completed by passing the condition \eqref{eq:nonnegscal-imcf} for $g_j$ to the limit as $j \to \infty$ exploiting \cref{eq:convwill} and \cref{eq:lschawking}.
%Passing to the limit the condition \eqref{eq:nonnegscal-imcf} for $g_j$, as $j \to +\infty$, using \eqref{eq:convwill} and \eqref{eq:lschawking}, completes the proof. 
\end{proof}

A direct corollary of \Cref{prop:characterization} and \Cref{thm:c0-stability} is an IMCF-proof of Gromov's scalar curvature $\CS^0$-stability theorem in dimension $3$, under additional global assumptions.

\begin{corollary}
\label{cor:gromov}
Suppose that the $g_j$'s are smooth metrics satisfying all the assumptions in \Cref{thm:c0-stability}. Furthermore,
suppose that the sequence of $g_j$'s  converges as $j\to \infty$ in $\CS^0_{\mathrm{loc}}(M)$ to a \emph{smooth} metric $g$. Then each $(M,g_j)$ and $(M,g)$
%and $g$ are smooth. Assume, in addition to the assumptions of \cref{thm:c0-stability}, that $g_j$ and $g$ are smooth. Then, they 
have a nonnegative scalar curvature in the classical sense.
\end{corollary}
\printbibliography
\end{document}